\documentclass[12pt]{article}

\usepackage[active]{srcltx}
\usepackage{amsmath}
\usepackage{amsthm}
\usepackage{amsxtra}
\usepackage{bbm}

\usepackage{amsfonts,amssymb}
\usepackage{latexsym}

\newtheorem{Th}{Theorem}
\newtheorem{Prop}{Proposition}
\newtheorem{Lm}{Lemma}
\newtheorem{Co}{Corollary}

\newtheorem{Qu}{Question}

\theoremstyle{definition}
\newtheorem{Def}{Definition}

\newcommand{\aut}{\mathrm{Aut}}
\newcommand{\bij}{\mathrm{Bij}}
\newcommand{\supp}{\mathrm{supp}}

\newcommand{\st}{\mathrm{St}}

\newcommand{\sym}{\mathrm{Sym}}
\newcommand{\rist}{\mathrm{rist}}
\newcommand{\comm}{\mathrm{comm}}
\newcommand{\dd}{\mathrm{d}}
\newcommand{\id}{\mathrm{Id}}
\newcommand{\homeo}{\mathrm{Homeo}}

\newcommand{\eg}{\text{e.g.\,}}
\newcommand{\ie}{\text{i.e.\;\,}}
\newcommand{\etc}{\text{etc.}}

\begin{document}
\title{On irreducibility and disjointness of Koopman and quasi-regular representations of weakly branch groups.}
\author{ {\bf Artem Dudko}  \\
                    Stony Brook University, Stony Brook, NY, USA  \\
          artem.dudko@stonybrook.edu \\
         {\bf Rostislav Grigorchuk}\footnote{NSF grant DMS-1207699, NSA grant H98230-15-1-0328 and ERC AG COMPASP. } \\        Texas A\&M University, College Station, TX, USA  \\      grigorch@math.tamu.edu }

\date{}

\maketitle

\section{Introduction.}
The main resources of examples of unitary representations of a countable group are Koopman, quasi-regular and groupoid representations. In the present paper we will study the Koopman and quasi-regular representations corresponding to actions of weakly branch groups on rooted trees. Their relation to the groupoid representation is studied in the second paper of the authors \cite{DuGr-spec} on the subject of spectral properties.

Branch  groups  were introduced  by  the  second  author  in  \cite{Gr00}  and  play  important  role  in many  investigations  in group theory  and  around  (see \cite{BGS03}).  They  posses  interesting  and  often  unusual  properties.  Branch just infinite groups constitute  one  of  three  classes  on  which  the  class  of  just  infinite  groups  (i.e.  infinite  groups  whose proper quotients  are finite)  naturally   splits. The class of branch groups  contains  groups  of  intermediate  growth, amenable  but  not  elementary  amenable  groups,  groups  with  finite  commutator  width  \etc.   Weakly branch  groups is a natural generalization of  the  class  of  branch  groups and   keep  many  nice  properties  of  branch  groups  (for instance  absence  of  nontrivial  laws, see \cite{Ab}). Weakly branch groups  also  play  important  role  in  studies  in holomorphic  dynamics (see \cite{Nekr}) and  in the  theory  of  fractals  (see \cite{GNS15}).

Throughout this paper we will assume that $G$ is a countable group. To every subgroup $H$ of $G$ one can associate a quasi-regular representation $\rho_{G/H}$ acting in $l^2(G/H)$. In particular, given an action of $G$ on a set $X$ for every $x\in X$ one can consider the corresponding quasi-regular representation $\rho_x=\rho_{G/\st_G(x)}$, where $\st_G(x)$ is the stabilizer of $x$ in $G$. Spectral proberties of quasi-regular representations play important roles in the Random Walks on groups and Schreier graphs (see \eg \cite{Kes59}, \cite{GrigZuk97} and \cite{Grig11}). In the case when $H=\{e\}$ (\ie is the trivial subgroup) the quasi-regular representation coincides with the regular representation $\rho_G$ which is of a special importance. Quasi-regular representations naturally give rise to Hecke algebras and their representations (see \eg \cite{An87} and \cite{Rad}).

A group acting on a rooted tree $T$ is called weakly branch if it acts transitively on each level of the tree and for every vertex $v$ of $T$ it has a nontrivial element $g$ supported on the subtree $T_v$ emerging from $v$ (see \eg \cite{BGS03} and \cite{Grig11}). Using Mackey criterion of irreducibility of quasi-regular representations Bartholdi and the second author in \cite{BG} showed that quasi-regular representations corresponding to the action of a weakly branch group on the boundary of a rooted tree are irreducible. One of the results of the present paper is the following:
\begin{Th}\label{ThQRDist} Let $G$ be a countable weakly branch group acting on a spherically homogeneous  rooted tree $T$. Then for any $x,y\in \partial T$, where $\partial T$ is the boundary of the tree, such that $x$ and $y$ are from disjoint orbits the quasi-regular representations $\rho_x$ and $\rho_y$ are not unitary equivalent.
\end{Th}
\noindent In particular, we obtain a continuum of pairwise disjoint (we will use the word "disjoint" as a synonym to "not unitary equivalent") irreducible representations of a weakly branch group $G$. Using Theorem \ref{ThQRDist} we show the following:
 \begin{Th}\label{ThCentralizer} Let $G$ be a countable weakly branch group acting on a spherically homogeneous rooted tree $T$. Then the centralizer of $G$ in the group $\bij(\partial T)$ of all bijections of the boundary of $T$ onto itself is trivial.
 \end{Th}
 \noindent This gives triviality of centralizers of $G$ in such groups as $\homeo(\partial T)$, $\aut(\partial T)$ \etc\, (see Corollary \ref{CoCentr}). As the result, we obtain a new example of a dynamical system for which the generalized Ismagilov's conjecture fails (see \cite{Kos10}).

  Another important family of representations gives the Koopman representation $\kappa$ associated to  a dynamical system $(G,X,\mu)$, where $(X,\mu)$ is a measure space on which $G$ acts by measure class preserving transformations.
  %, Koopman representation $\kappa$ is acting in $L^2(X,\mu)$ and is given by the formula
 % $$(\kappa(g)f)(x)=\sqrt{\frac{\dd\mu(g^{-1}(x))}{\dd\mu(x)}}f(g^{-1}x),$$ where $\frac{\dd\mu(g^{-1}(x))}{\dd\mu(x)}$ is the Radon-Nikodym derivative.
 Such representations are rarely irreducible. In the case when $\mu$ is a $G$-invariant probability measure the subspace of constant functions in $L^2(X,\mu)$ is $\kappa(G)$-invariant. And moreover, the restriction $\kappa_0$ of $\kappa$ on the subspace  $L^2_0(X,\mu)\subset L^2(X,\mu)$ of functions with zero integral (orthogonal complement to constant functions) is also usually reducible. A few known exceptions are listed in \cite{Glas03}. Recently, attention to irreducibility problem of representation $\kappa_0$ was raised by Vershik in \cite{V11}.  One of the results of the present paper is constructing new examples of irreducible Koopman representations corresponding to actions with quasi-invariant measures.

 For a $d$-regular rooted tree $T$ its boundary $\partial T$ can be identified with a space of sequences $\{x_j\}_{j\in\mathbb N}$ where $x_j\in\{1,\ldots,d\}$. Let
  \begin{equation}\label{EqMP}\mathcal P=\{p=(p_1,p_2,\ldots,p_d):p_i> 0\;\;\text{for}\;\;i=1,2,\ldots,d\;\;\text{and}\;\;\sum\limits_{i=1}^dp_i=1\}\end{equation} be the set of all probability distributions on the alphabet $\{1,2,\ldots,d\}$ assigning positive probability to every letter and \begin{equation}\label{EqMP*}\mathcal P^{*}=\{p\in\mathcal P:p_i\neq p_j\;\;\text{for all}\;\;1\leqslant i< j\leqslant d\}.\end{equation} For $p\in\mathcal P$ denote by $\mu_p=\prod\limits_{\mathbb N}p$ the corresponding Bernoulli measure on $\partial T$.  Our main result is:
 \begin{Th}\label{ThBranchIrred} Let  $G$ be a subexponentially bounded countable weakly branch group acting on a regular rooted tree and $p\in\mathcal P^{*}$. Then the following holds: \begin{itemize}
 \item[$1)$] the Koopman representation $\kappa_p$ associated to the action of $G$ on  $(\partial T,\mu_p)$ is irreducible;
 \item[$2)$] this representation is not unitary equivalent to any of the quasi-regular representations $\rho_x,x\in\partial T$;
 \item[$3)$] Koopman representations associated to different $p\in\mathcal P^{*}$ are pairwise disjoint.
 \end{itemize}
\end{Th}
\noindent Here subexponentially bounded group means a group consisting of subexponentially bounded  automorphisms of $T$. The precise definition will be given in Subsection \ref{SubsecBranch}. Notice that Theorem \ref{ThBranchIrred} gives an additional continuum of pairwise disjoint irreducible representations of a weakly branch group. These representations are faithful ($\kappa_p(g)\neq\id$ if $g$ is not the group unit). Observe also that for the uniform distribution $u=(\tfrac{1}{d},\tfrac{1}{d},\ldots,\tfrac{1}{d})\in\mathcal P$ the corresponding Bernoulli measure $\mu_u$ is $G$-invariant and the Koopman representation $\kappa_u$ is a direct sum of countably many finite-dimensional representations (see \cite{BG}). In fact, the authors see a possibility of generalizing  Theorem 3 to the case of distributions $p\in \mathcal P\setminus \{u\}$ (\ie such that $p_i\neq p_j$ for some $1\leqslant i<j\leqslant d$). However, since this leads to further complications of already difficult proof in this paper we focus on distributions from $\mathcal P^{*}$.

%%%%%%%%%%%%%%%%%%%%%%%%%%%

\section{Preliminaries.}
In this section we give necessary preliminaries on groups acting on rooted trees, representation theory and related topics.

\subsection{Weakly branch groups.}\label{SubsecBranch}
Here we give a brief introduction to actions of groups on boundaries of rooted trees. We refer the reader to \cite{Grig11} for detailed definitions and properties of these actions.

A rooted tree is a tree $T$, with vertex set divided into levels $V_n$, $n\in\mathbb Z_+$, such that $V_0$ consists of one vertex $v_0$ (called the root of $T$), the edges are only between consecutive levels, and each vertex from $V_n$, $n\geqslant 1$ (we consider infinite trees), is connected by an edge to exactly one vertex from $V_{n-1}$ (and several vertices from $V_{n+1}$). A rooted tree is called spherically homogeneous if each vertex from $V_n$ connected to the same number $d_n$ of vertices from $V_{n+1}$. $T$ is called $d-$regular, if $d_n=d$ is the same for all levels.

%%%%%%%%%%%%%%%%%%%%%%%%%%%%%%%
 An automorphism of a rooted tree $T$ is any automorphism of the graph $T$ preserving the root. The boundary $\partial T$ of $T$ is the set of all infinite paths starting at $v_0$ and passing through each level exactly one time. For a vertex $v$ of $T$  denote by $\partial T_v\subset\partial T$ the set of paths passing through $v$. Supply $\partial T$ by the topology generated by the sets $\partial T_v$. Automorphisms of $T$ act naturally on $\partial T$ by homeomorphisms. If $T$ is spherically homogeneous then $\partial T$ admits a unique $\aut(T)$-invariant measure $\mu$. This measure is uniform in the sense that
 $$\mu(\partial T_v)=\frac{1}{d_0d_1\ldots d_{n-1}}\;\;\text{for any}\;\;n\;\;\text{and any}\;\;v\in V_n.$$ Grigorchuk, Nekrashevich and Suschanski showed that this measure is ergodic if and only if the action of $G$ is transitive on each level $V_n$ of $T$ (level transitive). Moreover, in this case it is uniquely ergodic. We refer the reader to \cite{GNS00}, Proposition 6.5 for details.
\begin{Def}\label{DefBr} Let $T$ be a spherically homogeneous tree and $G<\aut(T)$. Rigid stabilizer of a vertex $v$  is the subgroup $\rist_v(G)=\{g\in G:\supp(g)\subset T_v\}$. Rigid stabilizer of level $n$ is
$$\rist_n(G)=\prod\limits_{v\in V_n}\rist_v(G).$$ $G$ is called \emph{branch} if it is transitive on each level and $\rist_n(G)$ is a subgroup of finite index in $G$ for all $n$. $G$ is called \emph{weakly branch} if it is transitive on each level $V_n$ of $T$ and $\rist_v(G)$ is nontrivial for each $v$.
\end{Def}

%The boundary $\partial T$ of a $d$-regular rooted tree is homeomorphic to a space of infinite sequences $\{1,2,\ldots,d\}^\mathbb{N}$ and hence is homeomorphic to a Cantor set. For a $d$-tuple $p=(p_1,\ldots,p_d)$ of positive numbers such that $p_1+\ldots +p_d=1$ define a measure $\nu_p$ on $\{1,2,\ldots, d\}$ by $$\nu_p(\{1\})=p_1,\nu_p(\{2\})=p_2,\ldots,\nu_p(\{d\})=p_d.$$ Let $\mu_p=\nu_p^\mathbb N$ be the corresponding Bernoulli measure on $\partial T$.
For each level $V_n$ of a binary rooted tree an automorphism $g$ of $T$ can be presented in the form
$$g=\sigma\cdot(g_1,\ldots,g_{d^n}),$$ where $\sigma\in\sym(V_n)$ is a permutation of the vertices from $V_n$ and $g_i$ are the restrictions of $g$ on the subtrees emerging from the vertices of $V_n$.
 \begin{Def} An element $g\in\aut(T)$ is polynomially bounded (in the sense of Sidki \cite{Sid00}) if the number $k_n(g)$ of restrictions $g_i$ to the vertices of level $n$ not equal to identity automorphism is bounded by a polynomial of $n$. We will call $g$ subexponentially bounded if for every $0<\gamma<1$ one has
 $$\lim\limits_{n\to\infty}k_n(g)\gamma^n=0.$$ A group $G<\aut(T)$ is polynomially (subexponentially) bounded if each $g\in G$ is polynomially (subexponentially) bounded. \end{Def}\noindent R. Kravchenko showed (see \cite{Kr}) that for any polynomially bounded  automorphism $g$ defined by a finite automaton and any $p\in\mathcal P$ (see \eqref{EqMP}) the measure $\mu_p$ is quasi-invariant with respect to the action of $g$. We will show in Section \ref{SecIrredKoop} that in fact the condition that $g$ is defined by a finite automaton is redundant and polynomial boundedness can be weakened to subexponential.
%%%%%%%%%%%%%%%%%%%%%%%%%%%%%%%%%%%%%%%%%%%%%%%%%%%%%%%%%%%%%%%%%%%%%%%%%%%%%%%%%

\subsection{Quasi-regular representations.}\label{SubsecQR}
 Given a countable group acting on a set $X$ and a point $x\in X$ one can define the quasi-regular representation $\rho_x$ in $l^2(Gx)$, where $Gx$ is the orbit of $x$,  by:
$$(\rho_x(g)f)(y)=f(g^{-1}y).$$ Notice that the isomorphism class of $\rho_x$ depends only on the stabilizer $\st_G(x)$ of $x$.

%%%%%%%%%%%%%%%%%%%%%%%%%%%%%
Recall that two subgroups $H_1,H_2$ of a group $G$ are called \emph{commensurable} if $H_1\cap H_2$
is of finite index in both $H_1$ and $H_2$. The groups $H_1$ and $H_2$ are called \emph{quasi-conjugate} in
$G$ if $gH_1g^{-1}$ is commensurable to $H_2$ for some $g\in G$. By definition,
\emph{commensurator} of $H<G$ is the subgroup of $G$ defined by
$$\comm_G(H)=\{g\in G:H\cap gHg^{-1}\text{ has finite index in }H\text{ and }gHg^{-1}\}.$$ Mackey proved the following (see \cite{Mack}, Corollary 7):
\begin{Th}\label{ThMackey} $1)$ Let $H$ be a subgroup of an infinite discrete countable group $G$. Then
the quasi-regular representation $\rho_{G/H}$ is irreducible if and only if $\comm_G(H) = H$.\\
$2)$ Let $H_1,H_2$ be two subgroups of $G$ such that $\comm_G(H_i)=H_i$, $i=1,2$.
Then $\rho_{G/H_1}$ and $\rho_{G/H_2}$ are unitary equivalent if and only if
$H_1$ and $H_2$ are quasi-conjugate.
\end{Th}

%%%%%%%%%%%%%%%%%%%%%%%%%%%%%
In \cite{BG} the authors showed the following:
\begin{Prop}\label{PropCommSt} Let $G$ be a weakly branch group, $T$ be the corresponding spherically homogeneous rooted tree, $x\in\partial T$ and $\st_G(x)$ be its stabilizer. Then $\comm_G(\st_G(x))=\st_G(x)$.
\end{Prop}
\noindent Theorem \ref{ThMackey} together with Proposition \ref{PropCommSt} immediately imply:
\begin{Co}\label{CoQRIrred} For a weakly branch group $G$ and any $x\in\partial T$ the quasi-regular representation $\rho_x$ is irreducible.
\end{Co}
%%%%%%%%%%%%%%%%%%%%%%%%%%%%%%%%%%%%%%%%%%%%%%%%%%%%%%%%%%%%%%%%

\subsection{Koopman representation.}\label{SubsecKoopBranch}
The most natural representations that one can associate to a measure-preserving action of a group $G$ on a measure space $(X,\mu)$, where $\mu$ is a probabilty measure, is the Koopman representation $\kappa$ of $G$ in $L^2(X,\mu)$ acting by:$$(\kappa(g)f)(x)=f(g^{-1}x).$$ This representation is important due to the fact that the spectral properties of $\kappa$ reflect the dynamical properties of the action such as ergodicity and weak-mixing.

%%%%%%%%%%%%%%%%%%%%%%%%%%%%%%%%%%%
 It is known that for an ergodic action operators $\kappa(g)$ together with operators of multiplication by functions from $L^\infty(X,\mu)$ generate in the weak operator topology the algebra of all bounded operators on $L^2(X,\mu)$. The representation $\kappa$ has invariant subspace of constant functions on $X$.  However, it is natural to ask whether $\kappa$ is irreducible in the orthogonal complement of the constant functions in $L^2(X,\mu)$ ($\kappa$ is called \emph{almost
 irreducible} in this case). This question is discussed in E. Glasner's book \cite{Glas03} and is again raised in Vershik's paper \cite{V11} (Problem 4). The cases when the answer is positive are quite rare. Among few examples is the example of arbitrary dense subgroup of the group $\aut([0,1],\mu)$ of all measure preserving automorphisms of the unit segment with Lebesgue measure $\mu$, supplied by the weak topology (see \cite{Glas03}).

 %%%%%%%%%%%%%%%%%%%%%%%%%%%%%%%%
More generally, if the measure $\mu$ is only quasi-invariant one can still define the Koopman type representation using Radon-Nikodim derivative:
$$(\kappa(g)f)(x)=\sqrt{\frac{\dd\mu(g^{-1}(x))}{\dd\mu(x)}}f(g^{-1}x).$$  If the measure $\mu$ is not invariant then constant functions do not form an invariant subspace and Koopman representation can be irreducible in the whole $L^2(X,\mu)$. There are several examples of group actions with quasi-invariant measures known for which the Koopman representation is irreducible:
\begin{itemize}
\item[$1)$] actions of free non-commutative  groups on their boundaries (\cite{FTP83}, \cite{FTS94} and \cite{KuSt});
\item[$2)$] actions of lattices of Lie-groups (or algebraic-groups) on their Poisson-Furstenberg boundaries (\cite{CS91} and \cite{BC02});
\item[$3)$] action of the fundamental group of a compact negatively curved manifold on its boundary endowed with the Paterson-Sullivan measure class (\cite{BM11});
\item[$4)$] canonical actions of Higman-Thompson groups on segments endowed with Lebesgue measure (\cite{Garn12} and \cite{D15 Thompson}).
\end{itemize}
However, the general case is not well understood and search for new examples is a challenging problem. In the present paper we construct a new class of examples of irreducible Koopman representations (see Theorem \ref{ThBranchIrred}).

%%%%%%%%%%%%%%%%%%%%%%%%%%%%%%%%%%%%%%%
Observe that for any spherically homogeneous rooted tree and any $G<\aut(T)$ the Koopman representation of $G$ corresponding to the invariant probability measure $\mu$ on $\partial T$ is a direct sum of countably many finite dimensional irreducible subrepresentations, since for each $n$ the finite-dimensional subspace of functions constant on the subtrees emerging from vertices from $V_n$ is invariant under $G$ (see \cite{BG}).
 %One of the main results of the present paper is the following:
%\begin{Th}\label{ThBranchIrred} Let  $G$ be a polynomially bounded weakly branch group acting on a regular rooted tree and $p$ be as above such that $p_i$ are pairwise distinct. Then the Koopman representation associated to the action of $G$ on  $(\partial T,\mu_p)$ is irreducible.
%\end{Th}
Recall that for two groups $G<H$ the centralizer of $G$ in $H$ is defined by:
$$C_H(G)=\{h\in H:hg=gh\;\;\text{for all}\;\;g\in G\}.$$
%One of important examples of branch groups is the Grigorchuk group $\Gamma$ acting on the boundary of the binary rooted tree and generated by elements $a,b,c,d$ satisfying the following recursions: $$a=\sigma\cdot(\id,\id),\;b=(a,c),\;c=(a,d),\;d=(\id,b),$$ where $\id$ is the identical action. The group $\Gamma$ satisfies many remarkable properties and answers a number of open questions. For example, $\Gamma$ is an infinite finitely generated torsion group (answer to one of the Burnside problems whether a finitely generated group in which every element has finite order must necessarily be a finite group), is the first example of intermediate growth groups, (answer to Milnor's problem) and is
% amenable but not elementary amenable. We refer the reader to \cite{Gr00} for the details. One of the results of the present paper is the following
Notice that for any group $G$ acting on a rooted tree $T$ we have the following group inclusions:
\begin{align}\label{EqGrIncl}
G<\aut(T)<\aut(\partial T,\mu)<\widetilde{\aut}(\partial T,\mu)\;\;\text{and}\\ \aut(T)<\homeo(\partial T)<\aut(\partial T)<\bij(T),\end{align}
where $\homeo(T)$ is the group of all homeomorphism of $\partial T$ onto itself, $\widetilde{\aut}(\partial T,\mu)$ is the group of all measure class preserving automorphisms of $(\partial T,\mu)$, $\aut(\partial T)$ is the group of all Borel automorphisms of $\partial T$ and $\bij(T)$ is the group of all bijections of $\partial T$ onto itself.
One of the results of the present paper (Theorem \ref{ThCentralizer}) is related to the question formulated by Kosyak (see \cite{Kos10}, Conjecture 0.0.1). For a measure space $(X,\nu)$ and a measure class preserving map $f:X\to X$ denote by $f_*(\nu)$ the push-forward measure on $X$. We will use the symbol $\id$ for the identity operator.
\begin{Qu}[Kosyak]\label{ConjKosyak} In which cases for a measure preserving dynamical system $(G,X,\nu)$ irreducibility of the associated Koopman representation $\kappa$  is equivalent to the following two conditions:\\
$1)$ $f_*(\nu)\perp\nu$ for any $f\in C_{\widetilde{\aut}(X,\nu)}(G),f\neq\id$;\\
$2)$ $\nu$ is $G$-ergodic.
\end{Qu}
 As a Corollary of Theorem \ref{ThCentralizer} we obtain:
 \begin{Co}\label{CoCentr} Let $G$ be a weakly branch group acting on a spherically homogeneous rooted tree $T$. Then the centralizers of $G$ in the groups $\aut(T),\homeo(\partial T),\aut(\partial T,\mu),\widetilde{\aut}(\partial T,\mu)$ and $\aut(\partial T)$ are trivial.
 \end{Co}
 \noindent Here $\mu$ is the unique $G$-invariant probability measure on $\partial T$. Thus, for every weakly branch group the action of $G$ on $(\partial T,\mu)$ satisfies conditions $1)$ and $2)$ of Question  \ref{ConjKosyak}. Since the corresponding Koopman representation is reducible, this gives a class of dynamical systems for which conditions $1)$ and $2)$ are not sufficient to assure irreducibility of $\kappa$.
 %%%%%%%%%%%%%%%%%%%%%%%%%%%%%%%%

%%%%%%%%%%%%%%%%%%%%%%%%%%%%%%%%%%%%%%%%%%%%%%%%%%%%%%%%%%%%%%%%%%%%%%%%%%%%%%%%%%
\section{Stabilizers and centralizers of weakly branch group actions.}
In this section we will prove Theorems \ref{ThQRDist} and \ref{ThCentralizer}. For simplicity, we denote $\partial T$ by $X$ and $\partial T_v$ by $X_v$ for a vertex $v$ of $T$.
%Let's prove an auxiliary Lemma. For $n\in\mathbb N$ and $v_1,v_2\in V_n$ let $l(v_1,v_2)$ be the number of coinciding vertices on the paths joining the top vertex of $T$ with $v_1$ and $v_2$ correspondingly.
%\begin{Lm}\label{LmTrans} For any $n$ and any $v_1,v_2,w_1,w_2\in V_n$ such that $l(v_1,v_2)=l(w_1,w_2)$ there exists $g\in\Gamma$ such that $gv_1=w_1,gv_2=w_2$.
%\end{Lm}
\begin{Lm}\label{LmQC} Let $G$ be a weakly branch group acting on a spherically homogeneous rooted tree $T$. Then for every pair $x, y\in X$ of points from disjoint orbits the groups $\st_G(x)$ and $\st_G(y)$ are not quasi-conjugate in $G$.
\end{Lm}
\begin{proof} Assume that for some $x,y$ from disjoint orbits the groups $\st_G(x)$ and $\st_G(y)$ are quasi-conjugate in $G$. By conjugating one of the groups by an appropriate element of $g$ we can assume that $\st_G(x)$ and $\st_G(y)$ are commensurable (see Subsection \ref{SubsecQR}). Equivalently, the orbits $\st_G(x)y$ and $\st_G(y)x$ are finite. Let us show that the latter is false.

Choose sufficiently large $n$ such that for the vertices $v,w\in V_n$ with $x\in  X_v$, $y\in  X_w$ one has $v\neq w$. For any $k>n$ let $v_k\in V_k$ be the vertex such that $ X_{v_k}\ni x$. Since $G$ is weakly branch, there exists $g_k\in G,g_k\neq\id$ (where $\id$ is the trivial automorphism) such that $\supp(g_k)\subset  X_{v_k}$. Then there exists $m>k$ and a vertex $u\in V_m$ such that $g_ku\neq u$. Since $G$ is level transitive, there exists $h\in G$ such that $hu=v_m$. Set $h_k=hg_kh^{-1}$. Then $\supp(h_k)\subset  X_v$, so in particular $h_k\in\st_G(y)$, and $h_kv_m\neq v_m$, and thus $h_k x\neq x$.

Finally, construct inductively an increasing sequence $k_l$ such that $h_{k_l}x\notin  X_{v_{k_{l+1}}}$ for every $l$. Then $h_{k_l}x$ are pairwise distinct, which shows that $\st_G(y)x$ is infinite. This contradiction finishes the proof of Lemma \ref{LmQC}
\end{proof}
As a corollary using Mackey Theorem \ref{ThMackey} we obtain Theorem \ref{ThQRDist}.
%%%%%%%%%%%%%%%%%%%%
 \begin{proof}[{\bf Proof of Theorem \ref{ThCentralizer}}]
 %%%%%%%%%%%%%%%%%%%%
 Assume that the centralizer $\mathcal C$ of a weakly branch group $G$ in $\bij( X)$ is not trivial. Let $c\in\mathcal C$, $c\neq\id$. Let $x\in  X$ such that $cx\neq x$. Consider two cases:\\
 $a)$ $cx\in Gx$. Then $c$ preserves the orbit $Gx$ and the unitary operator $C:l^2(Gx)\to l^2(Gx)$ given by:
 $$(Cf)(y)=f(c^{-1}y)$$ commutes with the representation $\rho_x$. Since $\rho_x$ is irreducible, by Schur's Lemma (see \eg \cite{BeHaVa}, Proposition 2.2) $C$ is a scalar operator. But this is impossible, since $C\delta_{x}=\delta_{cx}$ and $\delta_{cx}$ is orthogonal to $\delta_x$.\\
 $b)$ $cx\notin Gx$. Then $c$ maps the orbit $Gx$ onto the orbit $Gcx$ and the corresponding unitary operator $C:l^2(Gx)\to l^2(Gcx)$ intertwines representations $\rho_x$ and $\rho_{cx}$, which is impossible since by Theorem \ref{ThQRDist} representations $\rho_x$ and $\rho_{cx}$ are disjoint.
 \end{proof}

%%%%%%%%%%%%%%%%%%%%%%%%%%%%%%%%%%%%%%%%%%%%%%%%%%%%%%%%%%%%%%%%%%%%%%%%%%%%%%%%
\section{Irreducibility of Koopman representations of weakly branch groups.}\label{SecIrredKoop}
%%%%%%%%%%%%%%%%%%%%%%%%%%%%%%%%%%%%%%%%%%%%%%%%%%%%%%%%%%%%%%%%%%%%%%%%%%%%%%%%
In this section we will prove Theorem \ref{ThBranchIrred}. Let $G$ be a weakly branch group acting on a rooted tree $T$, $X=\partial T$ and $p\in\mathcal P$ (see \eqref{EqMP}). First we need to show that the Koopman representation corresponding to the action of $G$ on $( X,\mu_p)$ is well-defined (\ie that the measure $\mu_p$ is quasi-invariant).

Let $ T_n$ be the finite subtree of $T$ composed from levels up to the $n$th. Observe that for each $n$ the finite group of automorphism $\aut(T_n)$ can be identified with a subgroup of $\aut(T)$ consisting of elements $g$ such that all the restrictions of $g$ on subtrees of $T$ emerging from vertices of $n$-th level are trivial. For $g\in\aut(T)$  we denote by $g^{(n)}\in \aut( T_n)$ the automorphism of $T_n$ induced by $g$. We consider $g^{(n)}$ as an element of $\aut(T)$.
 \begin{Prop}\label{PropPolBound} For any subexponentially bounded automorphism $g$ of a $d-$regular rooted tree $T$ and any $p\in\mathcal P$ the measure $\mu_p$ on $ X$ is quasi-invariant with respect to $g$.
 \end{Prop}
 \begin{proof} Let $g\in\aut(T)$ be subexponentially bounded and $p\in\mathcal P$. Denote by $A_n$ the set of vertices $v\in V_n$ such that the restriction $g_v$ is not equal to identity. Let $P=\max\{p_1,\ldots,p_d\}$. Set
 $M_n=\bigcup\limits_{v\in A_n} X_v.$ Then by the definition of subexponential boundedness one has:
 $$\mu_p(M_n)\leqslant P^n|A_n|\to 0\;\;\text{when}\;\;n\to\infty.$$ Clearly, for every $n$ the automorphism $g^{(n)}$ preserves the class of measure $\mu_p$. Observe that on  $ X\setminus M_n$ the automorphism $g$ coincides with $g^{(n)}$, and thus has well-defined Radon-Nikodym derivative on this set. It follows that $g$ has well-defined Radon-Nikodym derivative for almost all $x\in  X$ and hence is measure class preserving.
 \end{proof}
 \begin{Prop}\label{PropErg} For any $p\in\mathcal P$ and any subexponentially bounded group $G$ acting level-transitively on a $d$-regular rooted tree $T$ the measure $\mu_p$ on $ X$ is ergodic with respect to the action of $G$.
 \end{Prop}
 \begin{proof} Assume that there exists a $G$-invariant subset $A\subset T$ such that $0<\mu_p(A)<1$. Let $v,w\in V_n$ for some $n\in\mathbb N$. Then there exists $g\in G$ such that $gv=w$. By subexponential boundedness of $g$ for every $\epsilon>0$ there exists $m>n$, $k\in\mathbb N$ and a collection $v_1,\ldots,v_k\in V_m$ such that
 the restriction of $g$ on the subtree $T_{v_i}$ is trivial for every $i=1,\ldots,k$ and $$\mu_p(B_m)>1-\epsilon,\;\;\text{where}\;\;B_m=\bigcup\limits_{i=1}^k X_{v_i}.$$ In particular, $g$ coincides with $g^{(m)}$ on $B_m$ and hence the Radon-Nikodym derivative of $g$ is constant on $X_{v_i}$ for every $i$. It follows that $$\frac{\mu_p(g(A\cap X_{v_i}))}{\mu_p(A\cap X_{v_i})}=\frac{\mu_p(g(X_{v_i}\setminus A))}{\mu_p(X_{v_i}\setminus A)}.$$ Since $\epsilon>0$ is arbitrary using $G$-invariance of $A$ we obtain that
  \begin{equation}\label{EqMupXvw}\frac{\mu_p(X_w\setminus A)}{\mu_p(A\cap X_w)}=\frac{\mu_p(g(X_v\setminus A))}{\mu_p(g(A\cap X_v))}=\frac{\mu_p(X_v\setminus A)}{\mu_p(A\cap X_v)}.\end{equation}
Since $v,w\in V_n$ are arbitrary, the latter is equal to $\tfrac{\mu_p(X\setminus A)}{\mu_p(A)}=\tfrac{1}{\mu_p(A)}-1.$ It follows that for every clopen set $B\subset X$ one has:
$$\mu_p(B\setminus A)=\big(\tfrac{1}{\mu_p(A)}-1\big)\mu_p(A\cap B).$$ Recall that clopen subsets of $X$ (which is homeomorphic to a Cantor set) approximate all measurable subsets by measure with arbitrary precision. Choosing a clopen subset $B$ such that $$\mu_p(A\cap B)>\tfrac{1}{2}\mu_p(A)\;\;\text{and}\;\;\mu_p(B\setminus A)<\tfrac{1}{2}(1-\mu_p(A))$$ we obtain a contradiction. This finishes the proof.
 \end{proof}
The proof of Theorem \ref{ThBranchIrred} is based on several technical statements. The following statement is a generalization of Proposition 23 from \cite{DuGr14}. The proof is similar. For the reader's convenience we present it here.
\begin{Prop}\label{PropBrAbsNonfree} Let $G$ be a subexponentially  bounded countable weakly branch group acting on a $d-$regular rooted tree $T$, $p\in\mathcal P^{*}$ $($see \eqref{EqMP*}$)$ and $\mu_p$ be the corresponding Bernoulli measure on $ X$. For any clopen subset $A\subset X$ and any $\epsilon >0$ there exists $g\in G$ such
that $$\supp(g)\subset A\;\;\text{and}\;\; \mu_p(A\setminus\supp(g))<\epsilon.$$\end{Prop}
Let us prove an auxiliary combinatorial lemma.
\begin{Lm}\label{LmOnSubs} Let $n\in\mathbb{N}$ and let $H<\sym(n)$ be a subgroup
acting transitively on $\{1,2,\ldots,n\}$. Let $A$ be a subset
 of $\{1,\ldots,n\}$ such that for all $g\neq h\in H$ one has
 $|g(A)\Delta h(A)|\leqslant |A|$, where $|A|$ is the cardinality of $A$. Then $|A|>n/2$.
\end{Lm}
\begin{proof} Set $k=|A|$. Then for any two distinct elements $h, g\in H$ we have:
 $$|h(A)\cap g(A)|=\frac{1}{2}(|h(A)|+|g(A)|-|h(A)\Delta g(A)|)\geqslant k/2.$$
 For every $h\in H$ introduce a vector $\xi_h\in \mathbb{C}^n$ by:
 $$\xi_h=(x_1,\ldots,x_n),\;\;\text{where}\;\;x_i=\left\{\begin{array}{ll}
 0,&\text{if}\;\;i\notin h(A),\\1,&\text{if}\;\;i\in h(A).
 \end{array}\right.$$ Denote by $(\cdot,\cdot)$  the standard scalar product in $\mathbb C^n$ and by $\|\cdot\|$ the corresponding norm. Then for every $h\neq g\in H$ one has:  $\|\xi_h\|^2=k$ and $(\xi_h,\xi_g)=|h(A)\cap g(A)|\geqslant k/2$.
 We obtain:
 $$\|\sum\limits_{h\in H}\xi_h\|^2\geqslant \tfrac{k}{2}m(m+1),\;\;\text{where}\;\;m=|H|.$$
 On the other hand, the group $H$ acts on $\mathbb{C}^n$ by permuting coordinates such that $g(\xi_h)=\xi_{gh}$ for all $g,h\in H$. Using transitivity of $H$ and the fact that the vector
  $\sum\limits_{h\in H}\xi_h$ is fixed by $H$ we get:
  $$\sum\limits_{h\in H}\xi_h=(\tfrac{km}{n},\tfrac{km}{n},\ldots,\tfrac{km}{n}),\;\;
  \|\sum\limits_{h\in H}\xi_h\|^2=\tfrac{k^2m^2}{n}.$$ It follows that $km\geqslant n\tfrac{m+1}{2}$ and
  $k>\tfrac{n}{2}$.
\end{proof}

 \noindent Let $d\geqslant 2$ be the valency of the regular rooted tree $T$.
 The proof of Proposition
 \ref{PropBrAbsNonfree} is based on the following:
 \begin{Lm} If $G$ is weakly branch then for every vertex $v$ there exists $g\in G$ with $\supp(g)\subset X_v$ such
 that $\mu_p(\supp(g))\geqslant \tfrac{1}{d}\mu_p(X_v)$.
 % Here $X_v\subset X= X$  is the set of infinite paths passing through $v$.
 \end{Lm}
 \begin{proof}  First, let us introduce some notations. For an element $h\in G$ set $$l(h)=\max\{l:h\in\st_G(l)\}.$$ As before, for $h\in\aut(T)$ and $n\in\mathbb N$ symbol $h^{(n)}$ denotes the element of $\aut(T_n)$ induced by $h$. We will use the same symbol for the corresponding permutation of $V_n$. For a vertex $v\in V_n,n\in\mathbb N$ and $l\geqslant n$ set $V_l(v)=T_v\cap V_l$.
%  Observe that given $h_1, h_2$ and $m\leqslant\min\{d(h_1),d(h_2)\}$ one has
% \[A_m(h_1h_2)=A_m(h_1)\Delta A_m(h_2)\;\;(\text{symmetric difference}).\]  Indeed, by the condition on $m$ the
%elements $h_1,h_2$ act trivially on the $m$-th level. Therefor, if both $h_1$ and $h_2$ move a vertex $v$ of the $m
%+1$-st level then $h_1h_2$ does not move $v$.
% Notice also that for any  $g,h\in G$ and $m\leqslant d(g)$ one has:
% \[A_m(hgh^{-1})=s^m_h(A_m(g)).\]

 Fix a vertex $v$ of the tree. Since $G$ is weakly branch there exist $g\neq \id$ such that
 $\supp(g)\subset X_v$. Set
 $$L=\min\{l(g):g\in G,g\neq \id,\supp(g)\subset X_v\}.$$
   For $g\in G$ denote by $W(g)$ the set of vertices $w$ from
 $V_L(v)$ such that $g$ induces a nontrivial permutation on $V_{L+1}(w)$. Set
 $$k(g)=|W(g)|,\;\;K=\max\{k(g):g\in G,\supp(g)\subset X_v\}.$$ By the choice of $L$, $K>0$. Fix an element $g\in G$ with $\supp(g)\subset X_v$
 such that $k(g)=K$.

Further, since $G$ acts transitively on $V_L$, we can find a number $m$
 and a collection of elements $H=\{h_1,h_2,\ldots,h_m\}\subset G$ such that the family $S=
 \{h_1^{(L)},\ldots,h_m^{(L)}\}$ of transformations of $V_L$
 forms a group preserving $V_L(v)$ and transitive on
 $V_L(v)$. Denote $g_i=h_igh_i^{-1}$. One has:
 $$W(g_i)=h_i^{(L)}(W(g)),\;\;W(g_ig_j)\supset
 W(g_i)\Delta W(g_j)$$ for all $i,j$. It follows that the set $W(g)$ together with the action of the group $S$ restricted
 to $V_L(v)$ satisfy the conditions of Lemma \ref{LmOnSubs}.
 Therefore, $K=|W(g)|> \tfrac{1}{2}|V_L(v)|$.

 For each $w\in W(g)\subset V_L(v)$ the element $g$
 induces a nontrivial permuation of $V_{L+1}(w)$, and thus
there exists at least two vertexes $w_1,w_2\in V_{L+1}(w)$ such that $$\supp(g)\supset X_{w_1}\cup X_{w_2}.$$ Set $\mathcal W=\{w_1,w_2:w\in W(g)\}.$ We have: $$|\mathcal W|> \tfrac{1}{d}|V_{L+1}(v)|\;\;\text{and}\;\;\supp(g)\supset \bigcup\limits_{u\in\mathcal W}X_u.$$

Now, since $G$ is level transitive, we can find a number $r$
 and a collection of elements $\widetilde H=\{\tilde h_1,\tilde h_2,\ldots,\tilde h_r\}\subset G$ such that the family $\tilde S=
 \{\tilde h_1^{(L+1)},\ldots,\tilde h_r^{(L+1)}\}$ of transformations of $V_{L+1}$
 forms a group preserving $V_{L+1}(v)$ and transitive on
 $V_{L+1}(v)$. Denote $\tilde g_i=\tilde h_ig\tilde h_i^{-1}$. One has:
$$\supp(\tilde g_i)\supset \bigcup\limits_{u\in\mathcal W}X_{\tilde h_i^{(L+1)}(u)}.$$ Since $\widetilde S$ is a group acting transitively on $V_{L+1}(v)$, for every $u\in \mathcal W$ the multiset $\{\tilde h_i^{(L+1)}(u):i=1,\ldots,r\}$ contains every vertex $u_1\in V_{L+1}(v)$ equal number $\frac{r}{|V_{L+1}(v)|}$ times. It follows that $$\sum\limits_{i=1}^r \mu_p(\supp(\tilde g_i))\geqslant |\mathcal W|\frac{r}{|V_{L+1}(v)|}\mu_p(X_v)>\tfrac{r}{d}\mu_p(X_v).$$ Therefore, there exists $i$ such that  $\mu_p(\supp(\tilde g_i))>\tfrac{1}{d}\mu_p(X_v)$. This finishes the proof.
\end{proof}
\begin{proof}[\textbf{Proof of Proposition \ref{PropBrAbsNonfree}}] Let $A$ be any clopen set and $g_0=\id$.
Construct by induction elements $g_n\in G,n=0,1,2\ldots$ such that
$\supp(g_n)\subset A$ and $\mu_p(A\setminus\supp(g_n))\leqslant \big(\tfrac{d}{d+1}\big)^n$.
 If $g_n$ is constructed choose vertices $v_1,\ldots v_k$ such that $X_{v_j}$ are disjoint subsets
 of $A\setminus\supp(g_n)$ and $\sum\mu_p(X_{v_j})\geqslant \tfrac{d}{d+1}\mu_p(A\setminus\supp(g_n))$.
  Using the lemma construct  elements $h_1,\ldots h_k$ such that
  $\supp(h_j)\subset X_{v_j}$ and $\mu_p(\supp(h_j))\geqslant\tfrac{1}{d}\mu_p(X_{v_j})$.
 Set $g_{n+1}=g_nh_1h_2\ldots h_k$. Then $\mu_p(A\setminus \supp(g_{n+1}))
 \leqslant \tfrac{d}{d+1}\mu_p(A\setminus \supp(g_n))$, which finishes the proof.
\end{proof}
As before, let $\st_G(1)=\{g\in G:gv=v\;\;\text{for each}\;\;v\in V_1\}$ be the stabilizer of the first level of $T$. For each point $x\in X$ denote by $N_g(x)\in\mathbb{Z}_+\cup \{\infty\}$ the number of the  vertices $v$ on the path defined by $x$ such that  $g_v\notin \st_G(1)$. Let $ T_n$ be the finite subtree of $T$ composed from levels up to the $n$th. Observe that for each $n$ the finite group of automorphism $\aut(T_n)$ can be identified with a subgroup of $\aut(T)$ consisting of elements $g$ such that $g_v$ is trivial for every $v\in V_n$. For $g\in\aut(T)$  we denote by $g^{(n)}\in \aut( T_n)$ the automorphism of $T_n$ induced by $g$. We consider $g^{(n)}$ as an element of $\aut(T)$.
\begin{Co}\label{CoNg} For every clopen set $A$ with $\mu_p(A)>0$, any $\epsilon>0$ and $k\in\mathbb{N}$ there exists $g\in G$ such that $\supp(g)\subset A$ and $$\mu_p(\{x:N_{g}(x)\geqslant k\})> (1-\epsilon)\mu_p(A).$$
\end{Co}
\begin{proof} We prove the statement by induction on $k$. The base follows immediately from Proposition \ref{PropBrAbsNonfree}. Assume that for given clopen set $A$ and $\epsilon>0$ we have an element $g$ such that
$$\mu_p(\{x:N_{g}(x)\geqslant k\})> (1-\epsilon)\mu_p(A).$$ For $n$ large enough one has $\mu_p(\{x:N_{g^{(n)}}(x)\geqslant k\})> (1-\epsilon)\mu_p(A)$. Let
$$D=\{x:N_{g}(x)\geqslant k\},\;C=\{x:N_{g}(x)= k\},\;B=\{x\in C:N_{g^{(n)}}(x)= k\}.$$
  Fix $\delta>0$. Increasing  $n$ if necessary we can assume that $\mu_p(C\setminus B)<\delta$. Find a clopen set $\tilde B\subset A$ such that $\mu_p(B\Delta\tilde B)<\delta$. Using Proposition \ref{PropBrAbsNonfree} find an element $h$ with $\supp(h)\subset\tilde B$ such that $\mu_p(\tilde B\setminus\supp(h))<\delta$ and $h\in\st_G(n+1)$. Then one has:
$$N_{gh}(x)\geqslant k+1\;\;\text{for all }\;\;x\in S=(B\cap\supp(h))\cup (D\setminus (C\cup\supp(h))).$$ By construction, $\mu_p(D\setminus S)<4\delta$. It follows that for $\delta$ sufficiently small one has
$$\mu_p(\{x:N_{gh}(x)\geqslant k+1\})\geqslant \mu_p(S)> (1-\epsilon)\mu_p(A).$$
\end{proof}

 For $g\in\aut(T)$ introduce an element $\alpha(g)\in \aut(T)$ as follows. For arbitrary $x\in X$ let $v$ be the vertex of a minimal level  from the path defined by $x$ such that $g(v)\neq v$. Write $x=vy$. Set
$$\alpha(g)x=g(v)y,$$ where $g(v)y$ should be understood as a concatenation of a finite word $g(v)$ and an infinite word $y$.   Observe that for every $x\in  X$ $\alpha(g)x$ differs with $x$ at most by one letter. Define $\beta(g)=g\alpha(g)^{-1}$. Observe that for all $x$ such that $gx\neq x$ one has \begin{equation}\label{EqAlBe}N_{\alpha(g)}(x)=1,\;\;N_{\beta(g)}(x)=N_g(x)-1.\end{equation}

Recall that $p_i$ are pairwise distinct. Set
$$a=a(p)=\min\{\tfrac{p_i}{p_j}:p_i>p_j\},\;\;\gamma=\gamma(p)=\frac{2\sqrt a}{a+1}.$$
%Observe that for  any $g\in\aut( T)$ the sequence of operators $\pi(g^{(n)})$ converges weakly to $\pi(g)$.
%\begin{Lm}\label{LmgxiA0} Let $A$ be a clopen set and $\xi_A$ be the characteristic function of the set $A$. Let $g\in\aut( T)$, $g(A)=A$ and $N_g(x)=1$ for almost all $x\in A$. Then $$(\pi(g)\xi_A,\xi_A)\leqslant \gamma\mu(A).$$
%\end{Lm}
%\begin{proof} For every $n$ denote by $g_n\in\aut(T_n)$ the action on $T_n$ induced by $g$. Then $g_n$ converges uniformly to $g$ when $n\to\infty$. We will prove the lemma assuming that $g\in \aut(T_n)$ for some $n$. Then the general case can be obtained taking limit when $n\to\infty$.
%
%Let $v\in V_n$ such that $X_v\subset A$. Let $v_0=v,v_1,\ldots,v_{k-1}$ be the orbit of $v$. Since $g$ changes exactly one letter of almost every $x\in A$, there exists a level $m\leqslant n$ such that all letters of all $v_i$ except the $m$-th letter coincide. Let $i(w)$ be the $m-th$ letter of $w$. For every $j$ let $q_j$ be the weight from $\{p_1,\ldots,p_d\}$ corresponding to the $m$-th letter of $v_j$:
%$$q_j=p_{i(g(v_j))}.$$ We have: \begin{align*}\bigg(\pi(g)\sum\limits_{j=0}^{k-1}\xi_{v_j},\sum\limits_{j=0}^{k-1}\xi_{v_j}\bigg)=\sum\limits_{j=0}^{k-1}
%\sqrt{\tfrac{q_{j+1}}{q_j}}\|\xi_{v_j}\|^2=\\
%\sum\limits_{j=0}^{k-1}\sqrt{q_jq_{j+1}}\leqslant \gamma\sum\limits_{j=0}^{k-1}\frac{q_j+q_{j+1}}{2}=\gamma\sum\limits_{j=0}^{k-1}\mu_p(X_{v_j}).\end{align*}
%Summing up the last inequality over all orbits of $g$ in $V_n$ we obtain the inequality of Lemma \ref{LmgxiA0}.
%\end{proof}

\begin{Lm}\label{LmgxiA} Let $A$ be a clopen set and $\xi_A$ be the characteristic function of the set $A$. Let $g\in\aut(T)$, $g(A)=A$ and $N_g(x)\geqslant k$ for almost all $x\in A$, where $k\in\mathbb{N}$. Then $$(\kappa_p(g)\xi_A,\xi_A)\leqslant \gamma^k\mu_p(A).$$
\end{Lm}
\begin{proof}
%For $h\in\aut(T)$ let $h_n\in\aut(T_n)\subset\aut(T)$ be the action on $T_n$ induced by $h$. Then \begin{equation}\label{EqPiHnConv}w-\lim\limits_{n\to\infty}\pi(h_n)=\pi(h).\end{equation} Let $g$ be from the conditions of the lemma. For simplicity, we will assume that $g\in\aut(T_n)$ for some $n$. The general case then can be obtained using \eqref{EqPiHnConv} by taking limit when $n\to\infty$.
We will prove the lemma using induction by $k$. The base of induction $k=0$ is trivial. Assume that the statement is true for given $k$. Let $g$ be such that $N_g(x)\geqslant k+1$ for almost all $x\in A$. We will use the presentation $g=\alpha(g)\beta(g)$ (see \eqref{EqAlBe}). Let $v$ be any vertex such that $\alpha(g)v\neq v$ and $\alpha(g)w=w$ for the vertex $w$ adjacent to $v$ above $v$. Denote by $v_0=v,v_1,\ldots,v_{l-1}$ the orbit of $v$ under $\alpha(g)$, where $\alpha(g)v_j=v_{j+1}$ for $j<l-1$ and $\alpha(g)v_{l-1}=v_0$. Observe that $\beta(g)v_j=v_j$ for every $j$. From the induction assumption it follows that $$(\kappa_p(\beta(g))\xi_{v_j},\xi_{v_j})\leqslant \gamma^k\mu_p(X_{v_j}).$$  For every $j$ let $q_j$ be the weight from $\{p_1,\ldots,p_d\}$ corresponding to the last letter of $v_j$. Then
$$\frac{\dd\mu_p(\alpha(g)^{-1}(x))}{\dd\mu_p(x)}=\frac{q_{j-1}}{q_j}$$ for all $x\in X_{v_j}$.
We obtain:
\begin{align*}\bigg(\kappa_p(g)\sum\limits_{j=0}^{l-1}\xi_{v_j},\sum\limits_{j=0}^{l-1}\xi_{v_j}\bigg)=
\sum\limits_{j=0}^{l-1}
\sqrt{\tfrac{q_{j-1}}{q_j}}(\kappa_p(\beta(g))\xi_{v_j},\xi_{v_j})\leqslant
\gamma^k\sum\limits_{j=0}^{l-1}\sqrt{\tfrac{q_{j-1}}{q_j}}\mu_p(X_{v_j})=\\ \gamma^k\sum\limits_{j=0}^{l-1}\sqrt{q_jq_{j-1}}\mu_p(X_{w})\leqslant\gamma^{k+1}
\sum\limits_{j=0}^{l-1}\tfrac{q_j+q_{j-1}}{2}\mu_p(X_v)=\gamma^{k+1}\sum\limits_{j=0}^{l-1}\mu_p(X_{v_j}).\end{align*}
Summing the last inequality over all the orbits $\{v_j\}$ as above (finite or countable number) we obtain the desired inequality.
\end{proof}

For a unitary representation of a discrete group $\Gamma$ in a Hilbert space $H$ denote by $\mathcal M_\pi$ the von Neumann algebra generated by operators $\pi(g),g\in \Gamma$ (\ie the closure of linear combinations of operators $\pi(g),g\in G$ in the weak operator topology). Let $B(H)$ be the algebra of all bounded operators on $H$ and $$\mathcal M_\pi'=\{R\in B(H):QR=RQ\;\;\text{for all}\;\;Q\in \mathcal M_\pi\}$$  be the commutant of $\mathcal M_\pi$.
The following fact is folklore:
 \begin{Lm}\label{LmFixed} Let $\pi$ be a unitary representation of a discrete group $\Gamma$ in a Hilbert space $H$. Set $H_1=\{\eta\in H:\pi(g)\eta=\eta\;\;\text{for all}\;\;g\in \Gamma\}$. Then the orthogonal projection $P$ onto $H_1$ belongs to $\mathcal M_\pi$.
 \end{Lm}
 \begin{proof} Let $B\in\mathcal M_\pi'$. Then $$\pi(g)B\eta=B\pi(g)\eta=B\eta$$ for every $\eta\in H_1,g\in \Gamma$. This implies that $BH_1\subset H_1$ and so $BP=PBP$. Same argument shows that $B^{*}P=PB^{*}P$, where $^{*}$ stands for the operation of conjugation in $B(H)$. Conjugating the latter identity we obtain that $PB=PBP=BP$. By von Neumann Bicommutant Theorem (see \eg Theorem 2.4.11 in \cite{BR}) we get that $P\in \mathcal (M_\pi')'=\mathcal M_\pi$.
 \end{proof}
For an open set $A\subset X$ define \begin{align}\label{EqHA}\begin{split}G_A=\{g\in G:\supp(g)\subset A\},\\ \mathcal{H}_A=\{\eta\in\mathcal{H}:\pi(g)\eta=\eta\;\;\text{for all}\;\;g\in G_A\}.\end{split}\end{align} Let $P_A$ be the orthogonal projection onto $\mathcal H_A$.  Applying Lemma \ref{LmFixed} to the restriction of the representation $\kappa_p$ onto the subgroup $G_A$ we obtain\begin{Co}\label{CoPA} For any open subset $A\subset X$ one has  $P_A\in \mathcal M_{\kappa_p}$.\end{Co}
\begin{Prop}\label{PropHA} One has $$\mathcal{H}_A=\{\eta\in\mathcal{H}:\supp(\eta)\subset X\setminus A\}.$$
\end{Prop}
\begin{proof} Clearly, every $\eta$ with $\supp(\eta)\subset X\setminus A$ belongs to $\mathcal{H}_A$. Assume that $\mathcal{H}_A$ is strictly larger than the subspace of functions $\eta$ with $\supp(\eta)\subset X\setminus A$. Then there exists a unit vector $\eta\in\mathcal{H}_A$ such that $\supp(\eta)\subset A$. Fix such a vector.

For $n\in\mathbb N$ denote by $V_n(A)$ the set of vertices $v$ from $V_n$ such that $X_v\subset A$. Let $\epsilon>0$. Since locally constant functions are dense in $\mathcal H$, one can find a level $n$ and constants $\alpha_v,v\in V_n(A)$ such that $$\|\eta-\sum\limits_{v\in V_n(A)}\alpha_v\xi_{X_v}\|\leqslant \epsilon.$$  By Corollary \ref{CoNg} and Lemma \ref{LmgxiA} for each $v\in V_n(A)$ there exists a sequence of elements $g_{v,k}$ such that $$\supp(g_{v,k})\subset X_v\;\;\text{and}\;\;
\lim\limits_{k\to\infty}(\pi(g_{v,k})\xi_{X_v},\xi_{X_v})=0.$$ Set $h_k=\prod\limits_{v\in V_n(A)}g_{v,k}$. Then $\supp(h_k)\subset A$ and $$\lim\limits_{k\to\infty}(\pi(h_k)\sum\limits_{v\in V_n(A)}\alpha_v\xi_{X_v},\sum\limits_{v\in V_n(A)}\alpha_v\xi_{X_v})=0.$$
It follows that $$\limsup_{k\to\infty}|(\pi(h_k)\eta,\eta)|\leqslant 2\epsilon+\epsilon^2.$$ Taking $\epsilon$, for instance, to be $\tfrac{1}{3}$ we obtain a contradiction to the fact that $\eta\in \mathcal H_A$ is a unit vector. This finishes the proof.
\end{proof}
\begin{proof}[{\bf Proof of Theorem \ref{ThBranchIrred}}] $1)$ Proposition \ref{PropHA} means that for each open set $A\subset X$ the orthogonal projection $P_A$ onto $\mathcal H_A$ is the operator of multiplication by the characteristic function of $X\setminus A$. Observe that every function from $L^\infty(X,\mu_p)$ can be approximated arbitrarily well in $L^2$-norm by finite linear combinations of characteristic functions of open sets. This implies that for every $m\in L^\infty(X,\mu_p)$ the operator of multiplication by $m$
$$\mathcal H\to \mathcal H,\;\;f\to mf$$ can be approximated arbitrary well in the strong operator topology by finite linear combinations of projections $P_A\in\mathcal M_{\kappa_p}$ (see Corollary \ref{CoPA}), and thus belongs to the von Neumann algebra $\mathcal M_{\kappa_p}$ generated by operators $\kappa_p(g),g\in G$. This implies that $\mathcal M_{\kappa_p}$ contains operators of multiplication by all functions from $L^\infty(X,\mu_p)$. Since for an ergodic measure class preserving action of a group $G$ on a measure space $(X,\mu_p)$ the algebra generated by group shifts and multiplication by functions coincide with the algebra $B(\mathcal H)$ of all bounded operators on $\mathcal H$ (see \eg \cite{Tak3}, Corollary 1.6) we obtain that $\mathcal M_{\kappa_p}$ coincides with $B(\mathcal H)$. By Schur's Lemma (see \eg \cite{BeHaVa}, Theorem A.2.2) this implies irreducibility of $\kappa_p$.

$2)$ Let $x\in  X$. For an open subset $A$ of $ X$ denote by $\mathcal H^x_A$ the subspace of $l^2(Gx)$ analogous to \eqref{EqHA}, but corresponding to the representation $\rho_x$:
 $$\mathcal{H}^x_A=\{\eta\in l^2(Gx):\rho_x(g)\eta=\eta\;\;\text{for all}\;\;g\in G_A\}.$$ Let $P^x_A$ be the orthogonal projection onto $\mathcal H^x_A$. Assume that $\kappa_p$ and $\rho_x$ are unitary equivalent via intertwining isometry $$U:L^2( X,\mu_p)\to l^2(Gx),$$ that is $U\kappa_\pi(g)=\rho_x(g)U$ for every $g\in G$. Choose a sequence of open covers $A_n$ of the orbit $Gx$ such that $\mu_p(A_n)\to 0$ when $n\to\infty$. From the definition of orthogonal projections $P_A$ and subspaces $\mathcal H_A$ we obtain:
  $$UP_{A_n}U^{*}=P^x_{A_n}$$ for every $n$. Since $G$ is weakly branch for every $n$ and every $y\in Gx$ the set $\{gy:g\in G_{A_n}\}$ is infinite. This implies that $P_{A_n}^x=0$. However, Proposition \ref{PropHA} implies that $P_{A_n}\to \id$ weakly when $n\to\infty$. This contradiction shows that $\kappa_p$ and $\rho_x$ are disjoint (not unitary equivalent).

$3)$ Let $\tilde p\in\mathcal P^{*},\tilde p \neq p$. For an open subset $A$ of $ X$ denote by $\tilde{\mathcal H}_A$ the subspace of $L^2(X,\mu_{\tilde p})$ analogous to \eqref{EqHA}, but corresponding to the representation $\kappa_{\tilde p}$. Let $\tilde P_A$ be the orthogonal projection onto $\tilde{\mathcal H}_A$. Since $\mu_{\tilde p}$ is singular to $\mu_p$ there exists $A\subset X$ such that $\mu_p(A)=0$ and $\mu_{\tilde p}(A)=1$. Assume that $\kappa_p$ and $\kappa_{\tilde p}$ are unitary equivalent via intertwining isometry $$U:L^2( X,\mu_p)\to L^2( X, \mu_{\tilde p}).$$ Let $A_n$ be a sequence of open covers of $A$ such that $\mu_p(A_n)\to 0$ when $n\to\infty$. Since $A\subset A_n$ we have that $\mu_{\tilde p}(A_n)=1$ for every $n$. From the definition of orthogonal projections $P_A$ and subspaces $\mathcal H_A$ we obtain:
  $$UP_{A_n}U^{*}=\tilde P_{A_n}=\id$$ for every $n$. But from Proposition \ref{PropHA} we obtain that $P_{A_n}\to 0$ weakly when $n\to\infty$. This contradiction finishes the proof.\end{proof}

% For $n\in\mathbb N$ set $$\mathcal L_n=\Sp\{\xi_{X_v}:v\in V_n \}\subset L^2( X,\mu)$$ (the subspace of functions constant on subtrees emerging from vertices of $n$-th level). Observe that $\mathcal L_n$ is naturally identified with the space of functions on $V_n$. Denote by $P_n$ the orthogonal projection from $L^2( X,\mu)$ onto $\mathcal L_n$.

%%%%%%%%%%%%%%%%%%%%%%%%%%%%%%%%%%%%%%%%%%%%%%%%%%%%%%%
\subsection*{Acknowledgement} The authors acknowledge the support of the Swiss NSF. The authors acknowledge Maria Gabriella Kuhn for useful discussions.
%%%%%%%%%%%%%%%%%%%%%%%%%%%%%%%%%%%%%%%%%%%%%%%%%%%%%%%

%%%%%%%%%%%%%%%%%%%%%%%%%%%%%%%%%%%%%%

\end{document}